\documentclass[12pt,a4paper]{article}

\usepackage[utf8]{inputenc}
\usepackage[english]{babel}
\usepackage{geometry} 
\usepackage{titlesec}
\usepackage{hyperref}
\usepackage{amssymb}
\usepackage{amsmath}
\usepackage{amsthm}
\usepackage{graphicx}
\usepackage{xcolor}
\usepackage[all]{xy}
\usepackage{comment}

\newcommand{\metrica}[1]{g\left(#1\right)}

\title{Non-compact inaudibility of weak symmetry and commutativity via generalized Heisenberg groups}

\author{Teresa Arias-Marco and José Manuel Fernández-Barroso\\
	\small{Department of Mathematics,}\\
	\small{University of Extremadura,} \\
	\small{Badajoz, Spain}\\
	\small{ariasmarco@unex.es, ferbar@unex.es}}
\date{}

\newtheorem{theorem}{Theorem}[section]
\newtheorem{proposition}[theorem]{Proposition}

\theoremstyle{definition}

\theoremstyle{definition}
\newtheorem{remark}[theorem]{Remark}

\theoremstyle{definition}

\theoremstyle{definition}

\newcommand{\z}{\mathfrak{z}}
\newcommand{\vgot}{\mathfrak{v}}
\newcommand{\n}{\mathfrak{n}}

\begin{document}
	
	\maketitle
	
	\begin{abstract}
		Two Riemannian manifolds are said to be isospectral if there exists a unitary operator which intertwines their Laplace-Beltrami operator.
		In this paper, we prove in the non-compact setting the inaudibility of the weak symmetry property and the commutative property using an isospectral pair of 23 dimensional generalized Heisenberg groups.
	\end{abstract}
	
	\textbf{Keywords:} Laplace operator; Isospectral Riemannian manifolds; Weakly-symmetric spaces; Commutative spaces; g.o. spaces
	
	\textbf{MSC2020:}
	58J53; 53C25; 58J50.
	
	\section{Introduction}
	
	Locally symmetric spaces are Riemannian manifolds whose local geodesic symmetries are local isometries. \textit{Symmetric-like} spaces are Riemannian manifolds with a geometric or algebraic property that locally symmetric spaces satisfy. Weakly symmetric and commutative spaces are examples of symmetric-like spaces. Selberg in \cite{S.56} introduced the notion of \textit{weakly symmetric} spaces as those on which any two points can be interchanged by an isometry. In the same paper, Selberg proved that the algebra of all invariant differential operators is commutative in every weakly symmetric space, these spaces are known as \textit{commutative} spaces. 
	
	Kaplan, in \cite{K.81}, developed examples of non locally symmetric Riemannian manifolds, namely \textit{generalized Heisenberg groups}. The study of symmetric-like properties in generalized Heisenberg groups has been studied by many authors (e.g. \cite{Rie.84},\cite{R.85},\cite{BTV.95},\cite{BRV.98} ).
	
	Two Riemannian manifolds, $M$ and $M'$ are said to be \textit{isospectral} if there exists a unitary operator $T:L^2(M')\to L^2(M)$ which intertwines their Laplacians, $\Delta$ and $\Delta'$, respectively, this is such that $T\circ\Delta'=\Delta\circ T$. If $M$ and $M'$ are compact Riemannian manifolds, the definition of being isospectral is equivalent to the fact that both manifolds have the same spectrum. In particular, Gordon and Wilson in \cite{GW.97} studied the isospectrality between pairs of generalized Heisenberg groups. Szabó in \cite{Sz.99} introduced a new example of isospectral closed Riemannian manifolds which can be seen as compact submanifolds of generalized Heisenberg groups. Szabó's examples were also introduced independently by Gordon, Gornet, Schueth, Webb and Wilson in \cite{GGSWW.98}, using a different approach.
	
	A geometric property is said to be \textit{inaudible}, or it cannot be heard, when one can find isospectral Riemannian manifolds such that one of them satisfies that property and the other does not. In \cite{AS.10}, the authors proved the inaudibility of some symmetric-like properties using Szabó's examples. In this paper, we rely on the work done in \cite{GW.97}, and we use a pair of isospectral generalized Heisenberg groups of dimension 23 to develop our main result:
	
	\begin{center}
		\textit{The property of being weakly symmetric and the property of being commutative are inaudible properties on non-compact Riemannian manifolds.}
	\end{center}
	
	In Section \ref{sec:isospectrality}, we treat the isospectrality between generalized Heisenberg groups. For that, we first study the isospectrality between compact 2-step nilmanifolds (i.e. quotients of 2-step nilpotent Lie groups) whose coverings are generalized Heisenberg groups. Finally, in Section \ref{sec:final}, we remember the main results concerning the study of symmetric-like properties on generalized Heisenberg groups, and we prove our main result.
	
	\section{Isospectrality on generalized Heisenberg groups}\label{sec:isospectrality}
	
	A well studied kind of manifolds are generalized Heisenberg groups, which are particular examples of 2-step nilpotent Lie groups. These manifolds were introduced by Kaplan in \cite{K.81} (for a more detailed information, consult \cite{BTV.95}). In this section we show how they are constructed, and how they are  classified according to their dimension.
	
	Let consider $\vgot$ and $\z$ real vector spaces which are orthogonal with respect to an inner product $g$, and a linear map $j:\z\to\mathfrak{so(v)}$. Let denote by $\n$ the real vector space $\vgot\oplus\z$. Using this $j$ map, one can define a Lie bracket $[\cdot,\cdot]$ on $\n$ by
	\begin{equation}\label{eq:relacionfundamenta-j-corchete}
		\metrica{[X,Y],Z}=\metrica{j_{Z}X,Y},
	\end{equation}
	for each $X,Y\in\vgot$ and each $Z\in\z$. This Lie bracket makes $\n$ a 2-step nilpotent Lie algebra, this is $[\n,\n]\subset \z$ and $[\n,\z]=0$. Thus, $\z$ is also known as the center of $\n$. If we need to distinguish the Lie algebra or the Lie bracket depending on the $j$ map, we will denote it by $\n(j)$ and $[\cdot,\cdot]^j$, respectively.
	
	$N(j)$ is the 2-step nilpotent Lie group whose Lie algebra is $\mathfrak{n}(j)$, and together with the metric induced by the inner product $g$, that we denote it again by $g$, the pair $(N(j),g)$ is a Riemannian manifold. The usual Lie group exponential map $\exp^j:\mathfrak{n}(j)\to N(j)$ which satisfies the Cambell--Baker--Hausdorff formula
	$$
	\exp^j(X)\cdot\exp^j(Y)=\exp^j\left(X+Y+\frac{1}{2}[X,Y]^j\right),
	$$
	is a diffeomorphism due to $N(j)$ is simply-connected and 2-step nilpotent. We will denote $\exp$ instead of $\exp^j$, if there is no confusion with $j$.
	
	A 2-step nilpotent Lie algebra, $\n$, is said to be a \textit{generalized Heisenberg algebra} when, additionally, $j$ satisfies the following relation
	\begin{equation}\label{j2-relation}
		j_{Z}^2=-\|Z\|^2\textup{Id}_\mathfrak{v},
	\end{equation}
	for every $Z\in\z$. In that case, its connected, simply-connected 2-step nilpotent Lie group, $N(j)$, is a \textit{generalized Heisenberg group}, and $j$ is a map of Heisenberg type.
	
	Generalized Heisenberg algebras, $\n=\vgot\oplus\z$, can be classified using the classification of representations of Clifford algebras over the real vector space $\z$ attached to a negative definite quadratic form $-\|\z\|^2$. The classification can be found in \cite{ABS.64}, but we include it here:
	
	\begin{enumerate}\renewcommand{\labelenumi}{\it \roman{enumi})}
		\item If $\dim\mathfrak{z}\not\equiv3\mod4$, then the Clifford algebra $Cl(\mathfrak{z})$ has a unique irreducible Clifford module $\mathfrak{v}_0$ up to equivalence, and the generalized Heisenberg algebra is obtained up to isomorphism, taking the direct sum of $p$ times the irreducible Clifford module, this is, $\mathfrak{v}\cong(\mathfrak{v}_0)^p,p\geq1.$
		\item If $\dim\mathfrak{z}\equiv3\mod4$, then the Clifford algebra $Cl(\mathfrak{z})$ has two non-equivalent irreducible Clifford modules $\mathfrak{v}_1$ and $\mathfrak{v}_2$, up to equivalence. The generalized Heisenberg algebras are obtained up to isomorphism, taking $\mathfrak{v}\cong(\mathfrak{v}_1)^p\oplus(\mathfrak{v}_2)^q$, with $p\geq0,q\geq0,p+q\geq1$, which is isomorphic to $\mathfrak{v}\cong(\mathfrak{v}_1)^q\oplus(\mathfrak{v}_2)^p$ due to $\mathfrak{v}_1$ and $\mathfrak{v}_2$ have the same dimension. In this case, we name the generalized Heisenberg algebras as $\mathfrak{n}_{p,q}$.
	\end{enumerate}
	
	With this notation, $\mathfrak{v}$ is said to be \textit{isotypic} if its Clifford module structure is irreducible. When $\dim\mathfrak{z}\not\equiv3\mod4$, $\mathfrak{v}$ is trivially isotypic. If $\dim\mathfrak{z}\equiv3\mod4$, $\mathfrak{v}$ is isotypic when $p=0$ or $q=0$, with $p+q\geq1$.
	
	Note that, in the case when $\dim\z\equiv3\mod4$, if we fix non-negative integers $(p_1,q_1)$ and $(p_2,q_2)$ such that $p_1+q_1=p_2+q_2$, then we obtain generalized Heisenberg algebras of the same dimension, $\n_{p_1,q_1}$ and $\n_{p_2,q_2}$. These algebras are isomorphic if and only if $(p_2,q_2)=\{(p_1,q_1),(q_1,p_1)\}$.
	
	Now, we want to find a pair of isospectral and non-isometric 2-step nilmanifolds whose Riemannian coverings are generalized Heisenberg groups. We also show  the isospectrality between those Riemannian coverings which are non-compact.
	
	Following Wilson \cite{W.82} and Gordon \cite{G.93}, we say that a Riemannian manifold is a 2-step \textit{nilmanifold} if it is a quotient of a simply-connected 2-step nilpotent Lie group $N(j)$ by a possible trivial discrete subgroup $\Gamma(j)$,  together with a left invariant Riemannian metric $g$. We denote it by $N^j=(N(j)/\Gamma(j),g)$ where $\Gamma(j)=\exp^j(\mathcal{M}+\mathcal{L})$ with $\mathcal{M}$ and $\mathcal{L}$ cocompact lattices in $\mathfrak{v}$ and $\mathfrak{z}$, respectively. 
	
	Wilson \cite{W.82} proved that two simply-connected Riemannian 2-step nilmanifolds are isometric if and only if the associated metric Lie algebras are isomorphic. According to the classification of the generalized Heisenberg algebras in terms of the attached Clifford algebra, there are up to isomorphisms one generalized Heisenberg algebra in each dimension when $\dim\z\not\equiv 3\mod4$, and there are two Lie algebras when $\dim\z\equiv 3\mod4$. We will center our attention in the case when $\dim\z\equiv 3\mod4$.
	
	Gordon introduced in \cite{G.93} the following map in order to work with the isotypic and the non isotypic representation of $\mathfrak{v}$ at the same time. Let $A$ denotes the quaternion numbers $\mathbb{H}$ or the octonion numbers $\mathbb{O}$. Suppose that $\mathfrak{v}=(A)^p\oplus(A)^q, p,q\geq0,p+q\geq1, p,q\in\mathbb{N}$, and the center is the pure elements of $A$, $\mathfrak{z}=A^*$. The map $j^{(p,q)}:\mathfrak{z}\to\mathfrak{so(v)}$ is given by
	\begin{equation}\label{eq:aplicacionGordon}
		j^{(p,q)}_Z(X_1,\dots,X_p,X_{p+1},\dots,X_{p+q})=(Z\cdot X_1,\dots,Z\cdot X_p,X_{p+1}\cdot Z,\dots,X_{p+q}\cdot Z),
	\end{equation}
	where $Z\in\mathfrak{z},X_i\in A, i=1,\dots, p+q$, and $\cdot$ is the usual multiplication in $A$. This map is always of Heisenberg type independently of the values of $p$ and $q$.
	
	We denote by $\mathfrak{n}(p,q)$ the generalized Heisenberg algebra $\mathfrak{n}_{p,q}(j^{(p,q)})$, by $N(p,q)$ the generalized Heisenberg group $\exp(\mathfrak{n}(p,q))$, and by $\mathcal{M}_{p,q}+\mathcal{L}$ the lattice in $\n(p,q)$ spanned by its integer standard basis. The 2-step nilmanifolds of Heisenberg type obtained by taking the cocompact quotient of $N(p,q)$ by $\Gamma_{p,q}:=\exp(\mathcal{M}_{p,q}+\mathcal{L})$, with its left invariant Riemannian metric, are denoted by $N^{p,q}$. With this notation, Gordon proved in \cite{G.93} the following result.
	
	\begin{theorem}
		$N^{p,q}$ and $N^{p',q'}$ are isospectral if $p+q=p'+q'$. Moreover, they are locally isometric if and only if $(p',q')=\{(p,q),(q,p)\}$.
	\end{theorem}
	
	\begin{remark}
		This result allows us to find pairs of isospectral 2-step nilmanifolds constructed using generalized Heisenberg groups. One possible situation in the case $\dim\z\equiv 3\mod4$ is
		$$\xymatrix{ N(p,q)\ar[d]& N(p+q,0) \ar[d]\\ N^{p,q}\ar@{~}[r]& N^{p+q,0} }$$
		where $N(p,q)$ and $N(p+q,0)$ with $p\geq0,q\geq0,p+q\geq1$, are the Riemannian covering of $N^{p,q}$ and $N^{p+q,0}$, respectively, and $\xymatrix{N^{p,q}\ar@{~}[r]& N^{p+q,0} }$ means that $N^{p,q}$ and $N^{p+q,0}$ are isospectral and not locally isometric in the compact sense. In particular, if $A=\mathbb{H}$ (resp. $A=\mathbb{O}$), then $\dim\z=3$ (resp. $\dim\z=7$), and $N^{1,1}$ and $N^{2,0}$ are isospectral. Moreover, they are not locally isometric, due to their corresponding 11-dimensional (resp. 23-dimensional) generalized Heisenberg algebras are not isomorphic.
	\end{remark}
	
	Now, we want to study the isospectrality in the non-compact sense between the Riemannian coverings $N(p,q)$ and $N(p+q,0)$. Szabó \cite{Sz.99}, proved their isospectrality in the case when $A=\mathbb{H}$, by constructing the explicit unitary operator which intertwines their Laplacians. The purpose of the rest of the section is to prove their isospectrality also when $A=\mathbb{O}$.
	
	\begin{proposition}\label{prop:isospec-noncompact}
		Let consider $A=\mathbb{O}$. Then, $N(p+q,0)$ and $N(p,q)$, $p\geq0,q\geq0,p+q\geq1$, are isospectral in the non-compact sense. Equivalently, there exists an operator, $T:L^2(N(p,q))\to L^2(N(p+q,0))$, intertwining the Laplacians $\Delta^{(p,q)}$ and $\Delta^{(p+q,0)}$, this is
		\begin{equation}\label{eq:int-lap}
			T\circ \Delta^{(p,q)}=\Delta^{(p+q,0)}\circ T.
		\end{equation}
		
	\end{proposition}
	
	\begin{proof}
		The generalized Heisenberg groups with 7-dimensional center $N(p,q)$ and $N(p+q,0)$ are diffeomorphic to $\mathbb{R}^{8(p+q)}\times\mathbb{R}^7$. Then, it is enough to prove the existence of an operator $T:L^2(\mathbb{R}^{8(p+q)}\times\mathbb{R}^7)\to L^2(\mathbb{R}^{8(p+q)}\times\mathbb{R}^7)$ satisfying \eqref{eq:int-lap}. From now on, we write $L^2(N(p,q))$ or $L^2(N(p+q,0))$ instead of $L^2(\mathbb{R}^{8(p+q)}\times\mathbb{R}^7)$ to distinguish the space of functions where the operator is working on. In order to do that, we first study the structure of $\Delta^{(p,q)}$.
		
		Let consider the following lattice
		$$
		\mathcal{L}=\left\{\sum_{i=1}^7\alpha_i\mathbf{e}_i:\alpha=(\alpha_1,\dots,\alpha_7)\in\mathbb{Z}^7\right\},
		$$
		where $\{\mathbf{e}_i\}_{i=1}^7$ is the standard basis of $\mathbb{R}^7$. The quotient $\mathbb{R}^{8(p+q)}\times(\mathbb{R}^7/\mathcal{L})$ can be seeing as $\mathbb{R}^{8(p+q)}\times\mathbb{T}^7$, where $\mathbb{T}^7$ is the 7-dimensional torus whose volume is 1. Thus, $\mathbb{R}^{8(p+q)}\times\mathbb{T}^7$ is diffeomorphic to $N(p,q)/\Gamma$, where $\Gamma=\exp(\mathcal{L})$. We denote by $\tilde{\Delta}^{(p,q)}$ the Laplace-Beltrami operator associated with the Riemannian manifold $N(p,q)/\Gamma$.
		
		As before, we write $L^2(N(p,q)/\Gamma)$ or $L^2(N(p+q,0)/\Gamma)$ instead of $L^2(\mathbb{R}^{8(p+q)}\times\mathbb{T}^7)$ in order to understand where the operators are acting. To relate $\Delta^{(p,q)}$ with $\tilde{\Delta}^{(p,q)}$, we need to establish a relation between functions in $L^2(N(p,q)/\Gamma)$ and functions in $L^2(N(p,q))$. We shall use the Fourier transform given by
		$$
		\mathcal{F}(f)(X,Z)=\int f(X,\xi)e^{-i\langle\xi,Z\rangle}d\xi,\quad X\in\mathbb{R}^{8(p+q)},\; Z\in\mathbb{R}^{7},\; \xi\in \mathbb{T}^{7},
		$$
		where $f\in(L^1\cap L^2)(N(p,q)/\Gamma)$. It is well defined almost everywhere and it has a unique extension onto $L^2(N(p,q)/\Gamma)$. Indeed, for each $f\in L^2(N(p,q)/\Gamma)$, it is satisfied
		\begin{equation}\label{eq:relacionlaplacianosfourier}
			\Delta^{(p,q)}\mathcal{F}(f)=\mathcal{F}(\tilde{\Delta}^{(p,q)}f).
		\end{equation}
		Therefore, we only need to study the structure of $\tilde{\Delta}^{(p,q)}$. Let fix an element $Z^\alpha=\sum_{i=1}^7\alpha_i\mathbf{e}_i\in\mathcal{L}$, with $\alpha=(\alpha_1,\dots,\alpha_7)\in\mathbb{Z}^7$, then, the function
		$$
		E^\alpha(Z)=e^{2\pi i\cdot\langle Z^\alpha,Z\rangle},\ Z\in\mathbb{R}^7,
		$$
		is periodic in $\mathcal{L}$. Let define the functions in $L^2_\mathbb{C}(N(p,q))$ by
		\begin{equation}\label{eq:funcionesgeneradorasLc}
			\Phi^\alpha(X,Z)=\varphi(X)E^\alpha(Z)
		\end{equation}
		with $\varphi$ in the complex function space $L^2_\mathbb{C}(\mathbb{R}^{8(p+q)})$. In addition, they are also well defined on $N(p,q)/\Gamma$. The complex function space of $N(p,q)/\Gamma$ can be decomposed in the following way
		\begin{equation}\label{eq:suma-subespacios}
			L^2_\mathbb{C}(N(p,q)/\Gamma)=\bigoplus_{\alpha\in\mathbb{Z}^7}L^2_{\mathbb{C}^\alpha}(N(p,q)/\Gamma),
		\end{equation}
		where $L^2_{\mathbb{C}^\alpha}(N(p,q)/\Gamma)$ is the function space spanned by the functions of the form (\ref{eq:funcionesgeneradorasLc}). Moreover, the Laplace-Beltrami operator acts separately on each subspace $L^2_{\mathbb{C}^\alpha}(N(p,q)/\Gamma)$, and they are invariant by its action, thus, one can consider the eigenvalue problem separately on each subspace. Szabó in \cite{Sz.99} computed the Laplace-Beltrami operator on these subspaces, it is given by \footnote{Note that the equations \eqref{eq:laplaciano-alpha} and \eqref{eq:laplaciano} differ from the ones presented in \cite{Sz.99} in the coefficient $\frac{\dim\mathfrak{v}}{4}$.} \small
		\begin{equation}\label{eq:laplaciano-alpha}
			(\tilde{\Delta}_\alpha^{(p,q)}\Phi^\alpha)(X,Z)=\left(\Delta_\mathfrak{v}+2\pi i\cdot \left(j^{(p,q)}_{Z^\alpha}X\right)\bullet-4\pi^2\|Z^\alpha\|^2\cdot\left(1+\frac{\dim\mathfrak{v}}{4}\|X\|^2\right)\right)(\varphi(X))E^\alpha(Z),
		\end{equation}\normalsize
		where $\Delta_\mathfrak{v}$ is the Laplace-Beltrami operator in the Euclidean space $\mathfrak{v}\cong\mathbb{R}^{8(p+q)}$, and $\left(j^{(p,q)}_{Z}X\right)\bullet\varphi$ acts as a derivation for every $\varphi\in L^2_\mathbb{C}(\mathbb{R}^{8(p+q)})$, $X\in\mathfrak{v}$, $Z\in\mathfrak{z}\cong\mathbb{R}^{7}$, with respect to the vector field $j^{(p,q)}_{Z}X$, this is
		$$
		\left(j^{(p,q)}_{Z}X\right)\bullet \varphi(X)=\sum_{i=1}^{8(p+q)}\frac{\partial \varphi}{\partial x_i}(X)\cdot \metrica{j^{(p,q)}_{Z}X,\mathbf{u}_i},
		$$
		where $\{x_i\}_{i=1}^{8(p+q)}$ are the coordinate functions on $\mathfrak{v}$, $\{\mathbf{u}_i\}_{i=1}^{8(p+q)}$ is the standard basis of $\mathfrak{v}$, and $g$ is the inner product defined on $\mathfrak{n}(p,q)$ with 7-dimensional center. This derivation is also well defined on functions $f\in L^2_{\mathbb{C}^\alpha}(N(p,q)/\Gamma)$, by
		$$
		\left(j^{(p,q)}_{Z}X\right)\bullet f(X,V)=\sum_{i=1}^{8(p+q)}\frac{\partial f}{\partial x_i}(X,V)\cdot \metrica{j^{(p,q)}_{Z}X,\mathbf{u}_i}.
		$$
		We denote both of them with the same notation if there is not risk of confusion.
		Thus, using the orthogonal decomposition of $L^2_{\mathbb{C}^\alpha}(N(p,q)/\Gamma)$ given in \eqref{eq:suma-subespacios} and the expresion of the Laplace-Beltrami given on each subspace by \eqref{eq:laplaciano-alpha}, Szabó \cite{Sz.99} established the Laplace-Beltrami operator associated with $N(p,q)/\Gamma$
		\begin{equation}\label{eq:laplaciano}
			\tilde{\Delta}^{(p,q)}f(X,\xi)=\left(\Delta_\mathfrak{v}- i\cdot \left(j^{(p,q)}_{\xi}X\right)\bullet -\left(1+\frac{\dim\mathfrak{v}}{4}\|X\|^2\right)\|\xi\|^2\right)f(X,\xi),
		\end{equation}
		where $f$ is an arbitrary function in $L^2_{\mathbb{C}}(N(p,q)/\Gamma)$.
		
		In order to find $T$ satisfying \eqref{eq:int-lap}, it is important to note that $\tilde{\Delta}^{(p,q)}$ and $\tilde{\Delta}^{(p+q,0)}$ given by \eqref{eq:laplaciano} only differ in the term depending on the $j$ map, $\left(j^{(p,q)}_{\xi}(X)\right)\bullet$ and $\left(j^{(p+q,0)}_{\xi}(X)\right)\bullet$. Moreover, due to \eqref{eq:suma-subespacios} and \eqref{eq:laplaciano-alpha}, it is enough to obtain an operator intertwining $\left(j^{(p,q)}_{Z^\alpha}(X)\right)\bullet$ and $\left(j^{(p+q,0)}_{Z^\alpha}X\right)\bullet$ on each $L^2_{\mathbb{C}^\alpha}(N(p,q)/\Gamma)$.
		
		Let consider $\iota:\mathbb{R}^7\to\mathbb{O}^*$ the natural embedding of $\mathbb{R}^7$ into the pure octonion space, $\bar{X}$ is the conjugate of $X\in\mathbb{O}$, and $Z^\alpha$ a fixed unitary vector in $\mathcal{L}$. Let take $\nu^\alpha\in\mathbb{R}^7$ an unit vector orthogonal to $Z^\alpha$, this means that $\|\nu^\alpha\|=1$ and $\langle\nu^\alpha,Z^\alpha\rangle=0$, with the usual scalar product in $\mathbb{R}^7$, equivalently $\textup{Re}(\iota(\nu^\alpha)\cdot\iota(Z^\alpha))=0$. Thus, it is important to note that $\iota(\nu^\alpha)\cdot\iota(Z^\alpha)$ remains inside $\mathbb{O}^*$, and $\bar{X}=-X$, for every $X\in\mathbb{O}^*$. Additionally, for any two orthogonal pure octonions, $X$ and $Y$, then $X\cdot Y=-Y\cdot X$. With that setting, we can define
		$$
		\sigma_{(p,q)}^\alpha:\mathbb{O}^{p+q}\longrightarrow\mathbb{O}^{p+q},
		$$
		by
		\begin{equation}\label{eq:operadorcambioj}
			\small
			\sigma_{(p,q)}^\alpha(X_1,\dots,X_p,X_{p+1},\dots,X_{p+q}):=(X_1,\dots,X_p,\overline{X_{p+1}\cdot\iota(\nu^\alpha)},\dots,\overline{X_{p+q}\cdot\iota(\nu^\alpha)}).
		\end{equation}
		Note that
		$$
		\begin{aligned}
			\overline{Y\cdot\iota(Z^\alpha)\cdot\iota(\nu^\alpha)}&=\iota(\nu^\alpha)\cdot\iota(Z^\alpha)\cdot\overline{Y} \\
			&=\iota(Z^\alpha)\cdot(-\iota(\nu^\alpha))\cdot\overline{Y} \\
			&=\iota(Z^\alpha)\cdot\overline{Y\cdot\iota(\nu^\alpha)},
		\end{aligned}
		$$
		for every $Y\in\mathbb{O}$. Thus, for some $X\in\mathbb{O}^{p+q}$,
		$$\small
		\begin{aligned}
			\sigma_{(p,q)}^\alpha&(j^{(p,q)}_{\iota(Z^\alpha)}(X))=\sigma_{(p,q)}^\alpha(\iota(Z^\alpha)\cdot X_1,\dots,\iota(Z^\alpha)\cdot X_p,X_{p+1}\cdot \iota(Z^\alpha),\dots,X_{p+q}\cdot \iota(Z^\alpha)) \\
			&=(\iota(Z^\alpha)\cdot X_1,\dots,\iota(Z^\alpha)\cdot X_p,\overline{X_{p+1}\cdot\iota(Z^\alpha)\cdot\iota(\nu^\alpha)},\dots,\overline{X_{p+q}\cdot\iota(Z^\alpha)\cdot\iota(\nu^\alpha)}) \\
			&=(\iota(Z^\alpha)\cdot X_1,\dots,\iota(Z^\alpha)\cdot X_p,\iota(Z^\alpha)\cdot\overline{X_{p+1}\cdot\iota(\nu^\alpha)},\dots,\iota(Z^\alpha)\cdot\overline{X_{p+q}\cdot\iota(\nu^\alpha)}) \\
			&=j^{(p+q,0)}_{\iota(Z^\alpha)}(\sigma_{(p,q)}^\alpha(X)).
		\end{aligned}
		$$
		This is
		$$
		\sigma_{(p,q)}^\alpha\circ j^{(p,q)}_{\iota(Z^\alpha)}=j^{(p+q,0)}_{\iota(Z^\alpha)}\circ\sigma_{(p,q)}^\alpha.
		$$
		Therefore, the induced map
		$$
		\sigma_{(p,q)*}^\alpha:L^2_{\mathbb{C}^\alpha}(N(p,q)/\Gamma)\to L^2_{\mathbb{C}^\alpha}(N(p+q,0)/\Gamma)
		$$
		is defined, for each $X_1,\dots,X_{p+q}\in\mathbb{O}$ and $Z\in\mathbb{T}^7$, by
		$$
		\sigma_{(p,q)*}^\alpha f(X_1,\dots,X_{p+q},Z):=f(\sigma_{(p,q)}^\alpha(X_1,\dots,X_{p+q}),Z).
		$$
		Note that $\sigma_{(p,q)*}^\alpha (f\cdot g)=\sigma_{(p,q)*}^\alpha f \cdot \sigma_{(p,q)*}^\alpha g$, for every $f,g\in L^2_{\mathbb{C}^\alpha}(N(p,q)/\Gamma)$. Then, $\sigma_{(p,q)*}^\alpha$ intertwines $\left(j^{(p,q)}_{\iota(Z^\alpha)}X\right)\bullet$ and $\left(j^{(p+q,0)}_{\iota(Z^\alpha)}X\right)\bullet$, this is
		$$
		\begin{aligned}
			\left(j^{(p+q,0)}_{\iota(Z^\alpha)}X\right)\bullet \sigma_{(p,q)*}^\alpha (f(X,Z))&=\left(j^{(p+q,0)}_{\iota(Z^\alpha)}X\right)\bullet \left(\sigma_{(p,q)*}^\alpha f(X,Z)\right)\\
			&=\left(j^{(p+q,0)}_{\iota(Z^\alpha)}X\right)\bullet f(\sigma_{(p,q)}^\alpha (X),Z))\\
			&=\frac{\partial f}{\partial X}(\sigma_{(p,q)}^\alpha (X),Z)\cdot f \left(j^{(p+q,0)}_{\iota(Z^\alpha)}\left(\sigma_{(p,q)}^\alpha (X),Z\right)\right)\\
			&=\sigma_{(p,q)*}^\alpha\left(\frac{\partial f}{\partial X}(X,Z)\right)\cdot f\left(\sigma_{(p,q)}^\alpha \left(j^{(p,q)}_{\iota(Z^\alpha)}X\right),Z\right)\\
			&=\sigma_{(p,q)*}^\alpha\left(\frac{\partial f}{\partial X}(X,Z)\right)\cdot \sigma_{(p,q)*}^\alpha f\left(j^{(p,q)}_{\iota(Z^\alpha)}X,Z\right)\\
			&=\sigma_{(p,q)*}^\alpha\left(\frac{\partial f}{\partial X}(X,Z)\cdot f\left(j^{(p,q)}_{\iota(Z^\alpha)}X,Z\right)\right)\\
			&=\sigma_{(p,q)*}^\alpha\left(\left(j^{(p,q)}_{\iota(Z^\alpha)}X\right)\bullet f(X,Z)\right),
		\end{aligned}
		$$
		for every  $f(X,Z)\in L^2_{\mathbb{C}^\alpha}(N(p,q)/\Gamma)$.

		Now, due to $\sigma_{(p,q)*}^\alpha$ acts separately for each $\alpha$, it leaves each subspace $L^2_{\mathbb{C}^\alpha}(N(p,q)/\Gamma)$ invariant. Thus, we can define
		$$
		\sigma_{(p,q)*}:=\bigoplus_\alpha\sigma_{(p,q)*}^\alpha:L^2_{\mathbb{C}}(N(p,q)/\Gamma)\to L^2_{\mathbb{C}}(N(p+q,0)/\Gamma),
		$$
		such that
		$$
		\sigma_{(p,q)*}\circ \left(j^{(p,q)}_{\xi}(X)\right)\bullet=\left(j^{(p+q,0)}_{\xi}(X)\right)\bullet\circ\sigma_{(p,q)*}.
		$$
		Therefore, $\sigma_{(p,q)*}$ intertwines the Laplacians of $N(p,q)/\Gamma$ and $N(p+q,0)/\Gamma$, this is
		$$
		\sigma_{(p,q)*}\circ \tilde{\Delta}^{(p,q)}=\tilde{\Delta}^{(p+q,0)}\circ\sigma_{(p,q)*}.
		$$
		
		Let consider $T:L^2(N(p,q))\to L^2(N(p+q,0))$ given by
		$$
		T=\mathcal{F}\circ \sigma_{(p,q)*}\circ \mathcal{F}^{-1}.
		$$
		Then, for each $f\in L^2(N(p,q)/\Gamma)$, it is satisfied
		$$
		\begin{aligned}
			T\circ \Delta^{(p,q)}(\mathcal{F}(f))&=\mathcal{F}\circ \sigma_{(p,q)*}\circ\mathcal{F}^{-1}(\Delta^{(p,q)}(\mathcal{F}(f)))=\mathcal{F}\circ \sigma_{(p,q)*}\circ\mathcal{F}^{-1}(\mathcal{F}(\tilde{\Delta}^{(p,q)}f))\\
			&=\mathcal{F}\circ \sigma_{(p,q)*}(\tilde{\Delta}^{(p,q)}f)=\mathcal{F}(\tilde{\Delta}^{(p+q,0)}\circ \sigma_{(p,q)*}(f))\\
			&=\Delta^{(p+q,0)}(\mathcal{F}(\sigma_{(p,q)*}(f)))=\Delta^{(p+q,0)}(\mathcal{F}\circ\sigma_{(p,q)*}\circ\mathcal{F}^{-1}(\mathcal{F}(f)))\\
			&=\Delta^{(p+q,0)}(\mathcal{F}\circ\sigma_{(p,q)*}\circ\mathcal{F}^{-1})(\mathcal{F}(f))=\Delta^{(p+q,0)}\circ T(\mathcal{F}(f)).
		\end{aligned}
		$$
		Equivalently,
		$$
		T\circ \Delta^{(p,q)}=\Delta^{(p+q,0)}\circ T.
		$$
	\end{proof}
	
	\section{Main results and conclusions}\label{sec:final}
	
	Generalized Heisenberg groups provide examples of Riemannian manifolds with different geometric properties depending on their dimension.
	
	Let $(M,g)$ be a Riemannian manifold, $\nabla$ its Levi--Civita connection and the curvature tensor of $M$,
	$$
	R(X,Y)=[\nabla_X,\nabla_Y]-\nabla_{[X,Y]},\quad X,Y\in\mathfrak{X}(M).
	$$
	Fixing a point $p\in M$ and a vector $v\in T_pM$, let denote by $\gamma_v$ the geodesic starting in $p$ with initial direction $v$. Then, a \textit{local geodesic symmetry} is the diffeomorphism $s_p$ such that it fixes the point $p$ and involves the vector of the geodesics in a neighborhood of the point, this is $s_p\circ\gamma_v=\gamma_{-v}$.
	
	A Riemannian manifold is said to be \textit{locally symmetric} if the local geodesic symmetries are in addition isometries. Moreover, Selberg in \cite{S.56} introduced the notion of \textit{weak symmetry}. These spaces are those on which any two points can be interchanged by an isometry. Berndt, Prüfer and Vanhecke proved in \cite{BPV.95} that weak symmetry is equivalent to \textit{ray symmetry} introduced by Szabó in \cite{Sz.93}.
	
	Selberg proved in \cite{S.56} that the algebra of all invariant differential operators of a weakly symmetric space is commutative. This fact points out a new definition, a \textit{commutative} space is a Riemannian space whose algebra of all invariant differential operators are commutative. Now, a natural question arises: \textit{Is any commutative Riemannian manifold a weakly symmetric space?} This question was answered negatively by Lauret in \cite{L.98.1}, where he found examples of commutative spaces which are not weakly symmetric.
	
	Berndt, Ricci and Vanhecke in \cite{BRV.98} distinguished between weakly symmetric spaces in the \textit{broad sense}, this is the original definition of Selberg, and weakly symmetric spaces in the \textit{narrow sense}, where the two points interchanging isometries must belong to the identity component of the group of isometries.
	
	Weak symmetry and commutativity were studied in generalized Heisenberg groups by Bernd, Ricci and Vanhecke in \cite{BRV.98}, and by Ricci in \cite{R.85}, respectively. These authors classified the generalized Heisenberg groups satisfying these properties in terms of their dimension.
	
	\begin{theorem}[Ricci, \cite{R.85}]\label{theo:clasificacionCommutative}
		A generalized Heisenberg group is commutative if and only if
		\begin{enumerate}
			\item[1.] $\dim\mathfrak{z}=1,2$ or 3,
			\item[2.] $\dim\mathfrak{z}=5$ and $\dim\mathfrak{v}=8$,
			\item[3.] $\dim\mathfrak{z}=6$ and $\dim\mathfrak{v}=8$,
			\item[4.] $\dim\mathfrak{z}=7$ and:
			\begin{enumerate}
				\item[i)] $\dim\mathfrak{v}=8$,
				\item[ii)] $\dim\mathfrak{v}=16$, with $\mathfrak{v}$ isotypic, this is $\n=\n_{2,0}\cong\n_{0,2}$.
			\end{enumerate} 
		\end{enumerate}
	\end{theorem}
	
	\begin{theorem}[Berndt, Ricci, Vanhecke, \cite{BRV.98}]\label{theo:clasificacionWS}
		A generalized Heisenberg group is weakly symmetric in the broad sense if and only if it is commutative. Moreover, every weakly symmetric in the broad sense generalized Heisenberg group is weakly symmetric in the narrow sense except the case when $\dim\z=1$.
	\end{theorem}
	
	Following Kowalski and Vanhecke \cite{KV.91}, a \textit{g.o. space} is a Riemannian manifold $(M,g)$ whose geodesics are orbits of a one-parameter subgroup of the group of the isometries of $M$. Generalized Heisenberg groups which are g.o. spaces were studied and classified by Riehm in \cite{Rie.84}.
	
	\begin{theorem}[Riehm, \cite{Rie.84}]\label{theo:clasificaciongo}
		A generalized Heisenberg group is a Riemannian g.o. space if and only if
		\begin{enumerate}
			\item[1.] $\dim\mathfrak{z}=1,2$ or 3,
			\item[2.] $\dim\mathfrak{z}=5$ and $\dim\mathfrak{v}=8$,
			\item[3.] $\dim\mathfrak{z}=6$ and $\dim\mathfrak{v}=8$,
			\item[4.] $\dim\mathfrak{z}=7$ and:
			\begin{enumerate}
				\item[i)] $\dim\mathfrak{v}=8$,
				\item[ii)] $\dim\mathfrak{v}=16$, with $\mathfrak{v}$ isotypic, this is $\n=\n_{2,0}\cong\n_{0,2}$,
				\item[iii)] $\dim\mathfrak{v}=24$, with $\mathfrak{v}$ isotypic, this is $\n=\n_{3,0}\cong\n_{0,3}$.
			\end{enumerate} 
		\end{enumerate}
	\end{theorem}
	
	Note that, generalized Heisenberg groups provide an example of g.o. space which is not weakly symmetric nor commutative, namely the generalized Heisenberg group with $\dim\mathfrak{z}=7$, $\dim\mathfrak{v}=24$ and $\mathfrak{v}$ isotypic.
	
	In order to prove the main results, we combine Proposition \ref{prop:isospec-noncompact} with the classifications of generalized Heisenberg groups presented in Theorem \ref{theo:clasificacionCommutative}, Theorem \ref{theo:clasificaciongo}, and Theorem \ref{theo:clasificacionWS}.
	
	\begin{proposition}
		One cannot hear if a non-compact Riemannian manifold is commutative. 
	\end{proposition}
	\begin{proof}
		By Proposition \ref{prop:isospec-noncompact}, the 23-dimensional non-compact generalized Heisenberg groups $N(2,0)$ and $N(1,1)$ with 7-dimensional center are isospectral. The classification of generalized Heisenberg groups which are commutative spaces, Theorem \ref{theo:clasificacionCommutative}, establishes that $N(2,0)$ is a commutative space while $N(1,1)$ is not a commutative space.
	\end{proof}
	
	The next result was already proved by Gordon in \cite{G.96} using a result by Gordon and Wilson in \cite{GW.97}. Summarizing, its proof is as the previous one, but using Theorem \ref{theo:clasificaciongo}.
	
	\begin{proposition}\label{prop:inaudible-go}
		One cannot hear if a non-compact Riemannian manifold is a g.o. space.
	\end{proposition}
	
	Now, using again the same pair of isospectral manifolds together with Theorem \ref{theo:clasificacionWS}, we obtain the result about the weak symmetric property.  
	
	\begin{proposition}
		One cannot hear if a non-compact Riemannian manifold is weakly symmetric, neither in the broad sense nor in the narrow sense.
	\end{proposition}
	
	These three propositions point out the inaudibility of being a g.o. space, being a commutative space, and being a weakly symmetric space in the non-compact case using $N(2,0)$ and $N(1,1)$ with 7-dimensional center. However, we cannot guarantee if their compact quotients, $N^{2,0}$ and $N^{1,1}$, inherit the same global properties, and it is still open the audibility problem of these properties in the compact setting. We can only guarantee that $N^{2,0}$ is weakly locally symmetric while $N^{1,1}$ is not. This fact proves the inaudibility in the compact case of the weak local symmetry in dimension 23. Nevertheless, this local property was already proved to be inaudible by the first author and Schueth in \cite{AS.10}, using a pair of 10-dimensional closed isospectral manifolds introduced by Szab\'{o} in \cite{Sz.99}.
	
	The main results proved in this paper add two new properties to the list of inaudible Riemannian properties, being a weakly symmetric space and being a commutative space, only in the non-compact case. So, it is natural to ask what happen in the compact case. This is:
	\begin{center}
		\textit{Can one find examples of isospectral closed Riemannian manifolds proving the inaudibility of being weakly symmetric, commutative or a g.o. space in the compact case?}
	\end{center}
	Therefore, it is important to find new examples of isospectral non-isometric pairs of compact Riemannian manifolds.
	
	We would also like to remark that using the above mentioned classifications, we can only guarantee the inaudibility in the non-compact case of the weakly symmetric and the commutative properties in dimension 23. Moreover, using Theorem \ref{theo:clasificaciongo} and working in the same way as in the proof of Proposition \ref{prop:inaudible-go}, but using the 31-dimensional non-compact generalized Heisenberg groups $N(3,0)$ and $N(2,1)$ with 7-dimensional center, we also obtain the inaudibility in the non-compact case of the g.o. property in dimension 31. This means, we can only guarantee the inaudibility of the g.o. property in the non-compact case in dimension 23 and 31. Then, a second question arises:
	\begin{center}
		\textit{Can one find examples of isospectral Riemannian manifolds in another dimensions which prove the inaudibility of the same symmetric-like properties? 
		}
	\end{center}
	We hope to answer these questions in forthcoming papers using new examples of closed isospectral manifolds developed by the authors.
	
	\textbf{Author contributions:} All authors contributed equally to this research and in writing the paper.
	
	\textbf{Funding:} The authors are supported by the grants GR21055 and IB18032 funded by Junta de Extremadura and Fondo Europeo de Desarrollo Regional. 
	The first author is also partially supported by grant PID2019-10519GA-C22 funded by AEI/10.13039/501100011033.
	
	\textbf{Conflicts of Interest:} The authors declare no conflict of interest. The founders had no role in the design of the study; in the collection, analyses, or interpretation of data; in the writing of the manuscript, or in the decision to publish the results.
	
    \bibliographystyle{plain}
	\bibliography{AMFB23.bib}
\end{document}